%% file: aistats2021.tex
\documentclass[twoside]{article} 

\newcommand{\affiliationText}{\scriptsize{$^*$ Equal contribution\\ $^1$ Zuse Institute, Berlin.\\ $^2$ MILA and DIRO, Université de Montréal \\ $^3$ Samsung SAIT AI Lab, Montréal \\ $^\dagger$ Canada CIFAR AI Chair}}
\usepackage[preprint]{aistats2021} 
\input{utils/defs}

\input{utils/preamble}

\input{utils/maths_commands}

\usepackage[backend=biber,style=authoryear-icomp,citestyle=authoryear,maxcitenames=1,uniquelist=false]{biblatex}
\newcommand{\citet}[1]{\textcite{#1}}
\newcommand{\citep}[1]{\parencite{#1}}
\newcommand{\remove}[1]{}

\makeatletter

\newrobustcmd*{\parentexttrack}[1]{%
  \begingroup
  \blx@blxinit
  \blx@setsfcodes
  \blx@bibopenparen#1\blx@bibcloseparen
  \endgroup}

\AtEveryCite{%
  }

\makeatother

\begin{document}

%

%
\runningauthor{Thomas Kerdreux, Lewis Liu, Simon Lacoste-Julien, Damien Scieur}

\twocolumn[ 
\aistatstitle{Affine Invariant Analysis of Frank-Wolfe on Strongly Convex Sets}
\aistatsauthor{ Thomas Kerdreux$^{1,*}$ \And Lewis Liu$^{2,*}$ \And  Simon Lacoste-Julien$^{2,3,\dagger}$ \And Damien Scieur$^{3,*}$}
\aistatsaddress{} 
] 

\begin{abstract}
    It is known that the Frank-Wolfe (FW) algorithm, which is affine-covariant, enjoys accelerated convergence rates when the constraint set is strongly convex. However, these results rely on norm-dependent assumptions, usually incurring non-affine invariant bounds, in contradiction with FW's affine-covariant property. In this work, we introduce new structural assumptions on the problem (such as the directional smoothness) and derive an affine invariant, norm-independent analysis of Frank-Wolfe. Based on our analysis, we propose an affine invariant backtracking line-search. Interestingly, we show that typical backtracking line-searches using smoothness of the objective function surprisingly converge to an affine invariant step size, despite using affine-dependent norms in the step size's computation. This indicates that we do not necessarily need to know the set's structure in advance to enjoy the affine-invariant accelerated rate.
\end{abstract}

\section{Introduction}
Conditional Gradient algorithms, a.k.a.\ Frank-Wolfe (FW) algorithms~\citep{frank1956algorithm}, form a class of first-order methods solving constrained optimization problems such as
\begin{equation}\label{eq:opt_problem}
    \underset{x\in\mathcal{C}}{\text{min }} f(x).
\end{equation}
The schemes in this class decompose non-linear constrained problems into a series of linear problems on the original constraint set, \textit{i.e.}\ linear minimization oracles (LMO). They form a practical family of algorithms~\citep{jaggi2013revisiting, bojanowski2014weakly,alayrac2016unsupervised,seguin2016instance,peyre2017weakly,miech2018learning,lacoste2015sequential,courty2016optimal,paty2019subspace,luise2019sinkhorn}; however, many open questions remain in designing such optimal algorithmic schemes (\textit{e.g.\ } \citep{braun2016lazifying,kerdreux2018frank,braun2018blended,combettes2020boosting,carderera2020second,Mortagy20,Combettes20}) and in their theoretical understanding.

\begin{algorithm}[t]
  \caption{Frank-Wolfe Algorithm}\label{algo:FW_general}
  \begin{algorithmic}[1]
    \REQUIRE $x_0 \in \mathcal{C}$.
    \FOR{$k=0, 1, \ldots, K $}
    	\STATE 
    	$v_k \in \underset{v\in\mathcal{C}}{\text{argmax }} \langle -\nabla f(x_k), \, v - x_k\rangle \hfill \vartriangleright\text{LMO}$ \label{line:FW_LMO}
    	
        \STATE 
        $\gamma_k = \underset{\gamma\in[0,1]}{\text{argmin }} f(x_k + \gamma (v_k - x_k)) \hfill \vartriangleright\text{Line-search}$ \label{line:step_size}
        
        \STATE 
        $x_{k+1} = (1 - \gamma_t)x_{k} + \gamma_k v_{k} \hfill\vartriangleright\text{Convex update}$ \label{line:FW_update} 
    \ENDFOR
  \end{algorithmic}
\end{algorithm}

Besides, with the appropriate line-search, the iterates of the FW are \textit{affine covariant} under the affine transformation $y = Bx+b$ of problem \eqref{eq:opt_problem},
\begin{equation} \label{eq:opt_problem_affine_transformation}
    \min_{y \in \tilde{\mathcal{C}} = B^{-1} (\mathcal{C}-b)} \hspace{-4ex} \tilde f(y) \defas f(B^{-1}(y-b)), \quad B \text{ invertible.}
\end{equation}
\begin{definition} 
An algorithm is affine covariant when its iterates $(x_k)$ (resp. $(y_k)$) for problem \eqref{eq:opt_problem} (resp. \eqref{eq:opt_problem_affine_transformation}) satisfy
\[
    y_k = Bx_k+b.
\]
\end{definition}
In other words, the behavior of Algorithm~\ref{algo:FW_general} is insensitive to affine transformations or re-parametrization of the space. This means that, ideally, the theoretical rate for a affine covariant algorithm should be \textit{affine invariant}.

The original Frank-Wolfe algorithm (Algorithm \ref{algo:FW_general}) generally enjoy a slow sublinear rate $\mathcal{O}(1/K)$ over general compact convex set and smooth convex functions~\citep{jaggi2013revisiting}. In that setting, \cite{Clar10,jaggi2013revisiting} define a modulus of smoothness that leads to affine invariant analysis of the Frank-Wolfe algorithm, matching with the affine covariant behavior of the algorithm.

Many works have then sought to find structural assumptions and algorithmic modifications that accelerate this sublinear rate of $\mathcal{O}(1/K)$. The strong convexity of the set (or more generally uniform convexity, see \citep{kerdreux2020uniform}) is one of such structural assumptions which lead to various accelerated convergence rates, like linear convergence rates when the unconstrained optimum is outside the constraint set \citep{levitin1966constrained,demyanov1970,dunn1979rates,rector2019revisiting} or sublinear rates $\mathcal{O}(1/K^2)$ when the function is also strongly convex but without restrictions on the position of the optimum \citep{garber2015faster}. However, to the best of our knowledge, there exists no affine invariant analysis for these accelerated regimes
stemming from the strong convexity of the constraint set $\mathcal{C}$.

\begin{table*}[t]
    \centering
        \begin{tabular}{lcccccc} 
        \toprule
        Related Work & $\mathcal{C}$ & \hspace{-3ex}Str. cvx. $f$ \hspace{-4ex}~& ~$x^*$ & Algo & Step size & Rate \\
        \midrule
        \citet{Clar10}              & Simplex       & \xmark & Any                          & FW         & Scheduled        & $\mathcal{O}(1/K)$\\
        \citet{jaggi2013revisiting} & Convex   & \xmark & Any                          & FW         & Scheduled        & $\mathcal{O}(1/K)$\\
         \citet{lacoste2013affine} & Any           & \cmark & Interior                     & FW         & Exact ls         & Linear\\
        \begin{minipage}[c]{0.53\columnwidth}
         \citet{lacoste2015global} \\
         \citet{gutman2020condition}
         \end{minipage} 
         & Polytope      & \cmark & Any                          & Corr. FW   & Exact ls         & Linear\\
        \textcolor{darkorange}{\textbf{Our work}}           & \textcolor{darkorange}{Strongly cvx}  & \textcolor{darkorange}{\xmark} & \textcolor{darkorange}{$\nabla f(x^{\star})\neq 0$}  & \textcolor{darkorange}{FW}         & \textcolor{darkorange}{Backtracking ls}  & \textcolor{darkorange}{Linear} \\
                                    & \textcolor{darkorange}{Strongly cvx}  & \textcolor{darkorange}{\cmark} & \textcolor{darkorange}{Any}                          & \textcolor{darkorange}{FW}         & \textcolor{darkorange}{Backtracking ls}  & \textcolor{darkorange}{$\mathcal{O}(1/K^2)$} \\
        \bottomrule
        \end{tabular}
    \caption{Existing \textit{affine invariant} analysis of Frank-Wolfe for smooth convex functions under different schemes. 
    \newline
    \textbf{Strong convexity. } The strong convexity assumption is to be taken in a broad sense. In~\citep{lacoste2013affine,lacoste2015global}, the authors consider ``generalized geometric strong convexity'' (see their Eq.~39), an affine invariant measure of (generalized) strong convexity, while \citep{gutman2020condition} consider strongly convex functions relative to a pair $(\mathcal{C},\,\omega)$ where $\omega$ is a distance-like function. In our work, we do not directly assume strong convexity, but the \textit{directional smoothness} of the function (see later Definition~\ref{def:directionnal_smoothness}), whose constant is bounded if various assumptions are satisfied for problem~\eqref{eq:opt_problem} (Theorem \ref{thm:directionnal_smoothness_bound}). \newline
    \textbf{Step size. } By \emph{scheduled} step sizes, 
    we consider, for instance, the classical $\gamma_k = \frac{2}{k+2}$.
    We denote by \emph{exact-line search} when the optimal step size depends on an unknown affine invariant quantity, whose accessible upper-bounds are affine-dependent (thus breaking the affine invariance of FW).
    }
    \label{tab:affine_invariant_analyses}
\end{table*}

In these ``non affine invariant'' analyses, structural assumptions like the $L$-smoothness (Definition \ref{def:smoothness_fun}) of $f$ and the $\alpha$-strong convexity of $\mathcal{C}$ (Definition \ref{def:strong_convexity_set}) lead to accelerated convergence rate of the Frank-Wolfe algorithm, but are typically conditioned on parameters $L,\alpha$ and others, which depend on a particular choice of a norm. This is surprising given that the Frank-Wolfe algorithm (under appropriate line-search) does not depend on any norm choice. 

Recall that the smoothness of a function and the strong convexity of a set are defined as follows.
\begin{definition}\label{def:smoothness_fun}
The function $f$ is \textbf{smooth} over the set $\mathcal{C}$ w.r.t.\ the norm $\|\cdot\|$ if there exists a constant $L>0$ such that, for any $x\,,y \in\mathcal{C}$, we have
\begin{equation}\label{eq:smoothness_fun}
    f(y)\leq f(x) + \langle \nabla f(x), \, y-x\rangle + \frac{L}{2}\|x-y\|^2.
\end{equation}
\end{definition}
\begin{definition}\label{def:strong_convexity_set}
A set $\mathcal{C}$ is \textbf{$\alpha$-strongly convex} with respect to a norm $\|\cdot\|$ if, for any $(x,y)\in\mathcal{C}$, $\gamma\in[0,1]$ and $\|z\|\leq 1$, we have
\begin{equation}\label{eq:def_strong_convexity_set}
    \gamma x + (1-\gamma)y + \alpha\gamma(1-\gamma) \|x-y\|^2 z \in\mathcal{C}.
\end{equation}
\end{definition}

Obtaining \textit{practical} accelerated affine invariant rates is hard, as an affine invariant step size is required. Indeed, some adaptive step sizes rely on theoretical affine invariant quantities which are in general not accessible. Therefore, by practical, we consider rates that can be achieved without a deep knowledge of the problem structure and constants.

For instance, scheduled step sizes, \textit{e.g.\ }$\gamma_k=\frac{2}{k+2}$, makes the Frank-Wolfe algorithm practically affine covariant, yet they do not capture the accelerated convergence regimes.
Exact line-search guarantees a practically affine covariant algorithm while capturing accelerated convergence regimes but significantly increases the time to perform a single iteration. 
Finally, it is possible to  use backtracking line-search such as \citep{pedregosa2020linearly}. Unfortunately, backtracking techniques rely on the choice of a specific norm, thus breaking affine invariance of the algorithm. 
\begin{minipage}{\linewidth}
\vspace{0.7mm}
This raises naturally the following questions:
\begin{framed}
\begin{center}
    \textit{Can we derive affine invariant rates for the Frank-Wolfe algorithm on strongly convex sets?}
\end{center}
\end{framed}
\begin{framed}
\begin{center}
    \textit{Can we design an affine invariant backtracking line-search for Frank-Wolfe algorithms?}
\end{center}
\end{framed}
This work provides a positive answer to these questions, by proposing the following contributions.
\end{minipage}

\paragraph{Contributions.} In this paper, \textcolor{darkorange}{\textbf{1)}} we conduct affine invariant analyses of the Frank-Wolfe Algorithm \ref{algo:FW_general}, when the function $f$ is smooth w.r.t.\ to a specific distance function $\omega(\cdot)$ and the set $\mathcal{C}$ is strongly convex also w.r.t.\ $\omega(\cdot)$. We then introduce new structural assumptions extending the class of problems for which such accelerated regimes hold in the case of Frank-Wolfe, called \textit{directionally smooth functions with direction $\delta$}. From this definition, \textcolor{darkorange}{\textbf{2)}} we propose an affine invariant backtracking line-search for finding the optimal step size. Finally, \textcolor{darkorange}{\textbf{3)}} we show that existing backtracking line-search methods, which use a specific norm, converges surprisingly to the optimal norm-invariant, affine invariant step size, meaning that affine-dependent and affine invariant backtracking techniques perform similarly.

\paragraph{Outline.} In Section \ref{sec:affine_dependent_analysis_FW}, we motivate the need for affine invariant analysis of Frank-Wolfe on strongly convex sets. 
In Section~\ref{sec:assumptions} and~\ref{sec:directional_smoothness}, we introduce the structural assumptions on the optimization problem that we will consider for analysing Frank-Wolfe. In Section \ref{sec:affine_invariant_analysis} we detail our affine invariant analysis of Frank-Wolfe on strongly convex set. In Section~\ref{sec:practical_affine_invariant_backtracking} and \ref{sec:explanation_efficiency_ls} we provide a backtracking line-search that directly estimate the affine invariant quantities we developed and we explain how it relates with existing ones. We conclude in Section \ref{sec:experiments} with numerical experiments.

\paragraph{Related Work.}
Other linear convergence rates of Frank-Wolfe algorithms exists with affine invariant analysis. For instance, corrective variants of Frank-Wolfe exhibit (affine invariant) linear convergence rates when the constraint set is a polytope \citep{lacoste2013affine,lacoste2015global} 
and the objective function is (generally) strongly convex. See Table \ref{tab:affine_invariant_analyses} for a review of all affine invariant analyses of Frank-Wolfe algorithms.

These affine invariant analyses emphasize that there is no specific choice of norm to be made in Frank-Wolfe algorithms as well as there is no need for affine pre-conditionners. Frank-Wolfe algorithms are arguably \textit{free-of-choice} methods, \textit{i.e.} little needs to be known on the optimization problem's structures to obtain the accelerated regimes. 
This is in line with recent works showing that the Frank-Wolfe methods exhibit accelerated adaptive behavior under a variety of structural constraints of \eqref{eq:opt_problem} which depend on inaccessible parameters, \textit{e.g.\ }H\"olderian Error Bounds on $f$ \citep{kerdreux2018restarting,YiAdapativeFW,Rinaldi2020} or local uniform convexity of $\mathcal{C}$ \citep{kerdreux2020uniform}.

Affine invariant analyses introduce constants seeking to characterize structural properties without a specific choice of norm. This has then been the basis for works extending the accelerated convergence analysis to non-smooth or non-strongly convex functions \citep{pena2019generalized,gutman2020condition}, which then explore new structural assumptions on $f$.

\section{``Affine-dependent'' Analysis of FW}\label{sec:affine_dependent_analysis_FW}

It is known that when the function is \textit{smooth} (Definition \ref{def:smoothness_fun}), the set is \textit{strongly-convex} (Definition \ref{def:strong_convexity_set}) and the gradient is lower bounded $\|\nabla f(x)\|\geq c$ over the constraint set (i.e., the constraints are active), the Frank-Wolfe algorithm \ref{algo:FW_general} converges linearly \citep{levitin1966constrained,demyanov1970,dunn1979rates}, at rate
\begin{equation} \label{eq:linear_conv_afine_dep}
    f(x_k)-f_{\star} \leq \; \left(\textstyle 1-\frac{L}{2c\alpha}\right)^k \left(f(x_0)-f_{\star}\right).
\end{equation}
Note that assuming the gradient to be lower bounded means the constraints are tight, i.e., the solution of the unconstrained counterpart lies outside the set of constraints. However, the constants $L$, $\alpha$, and $c$ depend on the choice of the norm for the smoothness and the strong convexity. In contrast, the Frank-Wolfe algorithm and iterates do not depend on such a choice, due to its affine covariance. Therefore, the rate of Algorithm \ref{algo:FW_general} should be affine invariant.
Unfortunately, it is possible to show that the known theoretical analyses can be \textit{arbitrarily} bad in the case where the constants $L,\,c,\,\alpha$ depend on ``affine variant'' norms.

\begin{example} \label{eq:example_not_affine_invariant}
    Consider the projection problem
    \[
        \textstyle \min_x f(x)\defas \frac{1}{2}\|x-\bar{x}\|^2 \quad \text{such that}\;\; \frac{1}{2}\|x\|^2 \leq 1.
    \]
    In such case, we have that $L=1,\;\alpha = \frac{1}{\sqrt{2}}$ and $c = 1-\|\bar{x}\|$ ($L,\,\alpha$ and $c$ are defined according to the $\ell_2$ norm). However, if we transform the problem into $\min_y f(By)$, the new constants become
    \[
        \textstyle L = \sigma_{\max}(B),\;\; \alpha = \frac{\sigma_{\min}(B)}{\sqrt{2}\sigma_{\max}(B)},\;\; c = \sigma_{\max}(B)(1-\|x_0\|).
    \]
    Comparing the rate \eqref{eq:linear_conv_afine_dep} of the two problems, identical to the eyes of the FW algorithm, we have that
    \begin{align*}
        f(x_k)-f^{\star} & \textstyle \leq \left( 1-\frac{1}{\sqrt{2}(1-\|\bar{x}\|)} \right)^k \big(f(x_0)-f^{\star}\big), \\
        f(By_k)-f^{\star} & \textstyle \leq \left( 1-\frac{\kappa^{-2}(B)}{\sqrt{2}(1-\|\bar{x}\|)} \right)^k\big(f(x_0)-f^{\star}\big),
    \end{align*}
    where $\kappa(B) = \frac{\sigma_{\max}(B)}{\sigma_{\min}(B)}$ is the condition number of $B$. This means we can artificially make a large theoretical upper bound on the rate of convergence by using an ill-conditioned transformation (i.e., $\kappa(B)$ large). However, the speed of convergence of FW iterates are \textit{not affected} by any linear transformation (dues to their affine-covariance), therefore the upper bound will not be representative of the true rate of convergence of FW.
\end{example}



When the optimum is in the relative interior of any compact set $\mathcal{C}$, FW converges linearly when $f$ is strongly convex~\citep{guelat1986some,lacoste2013affine}. On the other hand, linear convergence on strongly convex sets does not require strong convexity of $f$ when the solution of the unconstrained problem lies outside the set \citep{demyanov1970}. Our paper hence focuses on extending the analysis where the unconstrained optimum is outside the constrain set \citep{demyanov1970}.


These two analysis cover most practical cases, but not the situation where the unconstrained optimum is close to the boundary of $\mathcal{C}$. A recent analysis on strongly convex sets of \citep{garber2015faster} is not restrictive w.r.t.\ the position of the unconstrained optimum but conservative (convergence rate of $\mathcal{O}(1/K^2)$). It is interesting as it not only deals with the (previously unknown) situation where the unconstrained optimum is on the boundary on $\mathcal{C}$, but also when it is arbitrarily close to it, leading to poorly conditioned linear convergence regimes. In Appendix \ref{app:affine_invariant_garber_analysis}, we  provide an affine invariant analysis of \citep{garber2015faster}.

\section{Smoothness and Strong Convexity w.r.t.\ General Distance Functions}\label{sec:assumptions}

The major limitation in the definition of smoothness of a function (Definition \ref{def:smoothness_fun}) and the strong convexity of a set (Definition \ref{def:strong_convexity_set}) is the presence of the norm in their definition, whose constants may be dependent on affine transformation of the space (see Example \ref{eq:example_not_affine_invariant}). Technically, the notion of norm in the definition of smoothness and strong convexity of a function can be extended to the concept of distance-generating function, for instance using Bregman divergence \citep{bauschke2017descent,lu2018relatively} or gauge functions \citep{d2018optimal}.

Although is it classical to use different distance-generating functions $\omega$ (that satisfies Assumption \ref{assum:distance_fun} below) to characterize the smoothness of a function, we are not aware of such analysis for strongly convex sets. We believe that such analysis may exist, but for completeness we propose here an extension of the strong convexity of a set w.r.t.\ a distance function $\omega$.

\begin{assumption} \label{assum:distance_fun}
The function $\omega(\cdot)$ satisfies
\begin{itemize}[noitemsep,topsep=0pt]
    \item  $\omega(x) = 0 \;\; \Leftrightarrow \;\; x = 0$,
    \item \textbf{Positivity:} $\omega(x) \geq 0$,
    \item \textbf{Triangular Inequality:} $\omega( x + y) \leq \omega(x) + \omega(y)$
    \item \textbf{Positive homogeneity:} $\omega(\gamma x) = \gamma \omega(x)$, $\gamma \geq 0$,
    \item \textbf{Bounded asymmetry:} $\max_{x} \frac{\omega(x)}{\omega(-x)} \leq \kappa_{\omega}$.
\end{itemize}
\end{assumption}
Since $\omega(x)$ is convex by the triangle inequality, we define the dual distance
\begin{equation} \label{eq:dual_dist}
    \omega_{*}(v) = \max_{x : \omega(x) \leq 1} \langle v,x\rangle.
\end{equation}

\begin{remark}
    Usually, extensions of smoothness of a function use Bregman divergences (see e.g.\  \citep{lu2018relatively,bauschke2017descent}). However, the assumption that the distance-generating function is positively homogeneous is crucial in our analysis, which is unfortunately, not satisfied for most Bregman divergences.
\end{remark} 

A typical example satisfying such assumptions are gauge functions, also called \textit{Minkowski functional},
\[
    \omega_{\mathcal{Q}}(v) \defas \argmin_{\tau\geq 0} \tau \quad \text{subject to} \;\; v \in \tau \mathcal{Q},
\]
where $0\in \text{int}\mathcal{Q}$. Such distance-generating function satisfies Assumption \ref{assum:distance_fun} if the set $\mathcal{Q}$ is convex and compact, and contains $0$ in its interior. Moreover, gauge functions are affine invariant.

Usually, most works using gauge function assume that the set $\mathcal{Q}$ is \textit{centrally symmetric} \citep{d2018optimal,molinaro2020}, which add the assumption that
\[
    \omega(x) = \omega(-x).
\]
In that case, the gauge function is a norm \parencite[Theorem 15.2.]{rockafellar1970convex}. Removing symmetry extends non-trivially the definition of strongly convex sets w.r.t.\ the distance function $\omega$. We now recall the definitions of smoothness and strong convexity of a function w.r.t.\ a distance function $\omega$.
\begin{definition}\label{def:smoothness_strong_convexity_general}
    A function $f$ is  smooth (resp.\ strongly convex) w.r.t.\ the distance function $\omega$ if, for a constant $L_{\omega}$ (resp. $\mu_{\omega}$), the function satisfies
    \begin{eqnarray}
        f(y) & \leq f(x) + \langle \nabla f(x),\, y-x \rangle + \frac{L_{\omega}}{2} \omega^2(y-x), \label{eq:minkowsky_smooth_fun}\\
        f(y) & \geq f(x) + \langle \nabla f(x),\, y-x \rangle + \frac{\mu_{\omega}}{2} \omega^2(y-x) \label{eq:minkowsky_strong_convex_fun}.
    \end{eqnarray}
\end{definition}


\begin{definition} \label{def:minkow_strong_convex_set}
    A set $\mathcal{C}$ is $\alpha_{\omega}$-strongly convex w.r.t.\ $\omega$ if, for any $(x,y)\in\mathcal{C}$ and $\gamma\in[0,1]$, we have
    \begin{align*}
        z_\gamma  +\alpha_{\omega} \gamma (1-\gamma) \frac{ (1-\gamma)\omega^2(x-y) + \gamma \omega^2(y-x)}{2} z \in \mathcal{C},
    \end{align*}
    where $z_{\gamma} = \gamma x + (1-\gamma)y$, for all $z$ such that $\omega(z)\leq 1$.
\end{definition}
This definition extends the one of strongly convex sets with a general distance function that may not be a norm, see for instance \citep{garber2015faster}.

With Definition \ref{def:minkow_strong_convex_set}, the level sets of smooth and strongly convex functions are also strongly convex sets when the function $\omega$ is used. Such results appear for instance in \citep{journee2010generalized} when $\omega$ is the $\ell_2$ norm.

\begin{lemma}[Strong Convexity of Sets]~\label{thm:level_set_minkow_fun}
    Let $f$ be a $L$-smooth and $\mu$-strongly convex function w.r.t.\ $\omega$. Then, the set
    \[
        \mathcal{C} = \{ x : f(x)-f_{\star} \leq R \}
    \]
    is $\alpha$-strongly convex w.r.t.\ $\omega$, with $\alpha = \frac{ \mu_{\omega} }{\kappa_\omega\sqrt{2L_{\omega}R}}.$
\end{lemma} 
We defer the proof in Appendix \ref{app:proofs_assymetric_gauges}. This result corresponds \textit{exactly} to the one of \parencite[Theorem 12]{journee2010generalized}, when we use $\omega = \|\cdot\|$.

\paragraph{Scaling Inequality.} All proofs of Frank-Wolfe methods on strongly convex sets leverage the same property. The \textit{scaling inequality} (equivalent to strong convexity of $\mathcal{C}$ \parencite[Theorem 2.1.]{goncharov2017strong})  crucially relates the Frank-Wolfe gap with $\|x_t-v_t\|^2$, see \textit{e.g.\ }\parencite[Lemma 2.1.]{kerdreux2020uniform}.
We extend the scaling inequality to strongly convex sets with generic distance functions.

\begin{lemma}[Distance Scaling Inequality]\label{lem:scaling_inequality}
Assume $\mathcal{C}$ is $\alpha_{\omega}$-strongly convex w.r.t.\ $\omega$. Then for any $x, v\in\mathcal{C}$ and $\phi\in N_{\mathcal{C}}(v)$ (normal cone), we have
\begin{equation}\label{eq:scaling_inequality}
    \langle \phi , \, v - x \rangle \geq \alpha_{\omega} \omega_*\big(\phi\big) \omega^2(v - x).
\end{equation}
In particular for any iterate $x_k$ of Frank-Wolfe and its Frank-Wolfe vertex $v_k$ (Line \ref{line:FW_LMO} in Algorithm \ref{algo:FW_general}), we have
\[
\langle -\nabla f(x_k) ; v_k - x_k \rangle \geq \alpha_{\omega} \omega_*\big(-\nabla f(x_k)\big) \omega^2(v_k - x_k).
\]
\end{lemma}
\begin{proof}
    We start with $v_{\phi} = \argmax_{v\in\mathcal{C}}\langle \phi;\; v \rangle$. Then, we use the definition of strong convexity of a set,
    \[
        \gamma x + (1-\gamma)v_{\phi} + \alpha_{\omega}\gamma(1-\gamma) D_\gamma z \in\mathcal{C} \quad \forall z : \omega(z) \leq 1.
    \]
    where $D_\gamma(x-y)\defas \frac{\gamma \omega^2(x-y) + (1-\gamma)\omega^2(y-x)}{2} $. Then, by optimality of $v_{\phi}$,
    \[
        \langle \phi;\; v_{\phi} \rangle \geq \langle \phi;\; \gamma x + (1-\gamma)v_{\phi} + \alpha_{\omega}\gamma(1-\gamma) D_\gamma(x-v_{\phi}) z \rangle
    \]
    After simplification,
    \[
        \langle \phi;\; v_{\phi}-x \rangle \geq \alpha_{\omega}(1-\gamma) D_\gamma(x-v_{\phi}) \langle \phi;\;  z \rangle
    \]
    which holds in particular when $\phi = -\nabla f(x)$, $\gamma = 0$ and $z$ being the argmax (see \eqref{eq:dual_dist}).
\end{proof}

\section{Directional Smoothness}\label{sec:directional_smoothness}
We separately introduced smoothness for functions, and strong convexity for sets w.r.t.\ a distance function $\omega$.
Analyses of Frank-Wolfe algorithm on strongly convex sets \citep{levitin1966constrained,demyanov1970,dunn1979rates} show that, when $f$ is convex and smooth, and the unconstrained minima of $f$ are outside of $\mathcal{C}$, there is linear convergence.

We hence propose a novel condition that mingles the smoothness of $f$ with the strong convexity of $\mathcal{C}$ when moving in a specific direction $\delta$. We are interested in particular with the FW direction and we will see later that this assumption guarantees a linear convergence rate in this case. We call this condition the \textit{directional smoothness}.
\begin{definition} \label{def:directionnal_smoothness}
    The function $f$ is \textit{directionally smooth} with direction function $\delta :\mathcal{C} \rightarrow \mathbb{R}^d$ if there exists a constant $\mathcal{L}_{f,{\delta}}>0 $ s.t.\ $\forall x \in \mathcal{C}$ and $h>0$ with $x+ h\delta(x)\in\mathcal{C}$,
    \begin{align} \label{eq:directionnal_smoothness}
        f\big(x + h\delta(x)\big) \leq & f(x) - h\langle -\nabla f(x),\, \delta(x) \rangle \\
        & + \frac{\mathcal{L}_{f,\delta} h^2}{2}\langle - \nabla f(x),\, \delta(x) \rangle.\nonumber
    \end{align}
\end{definition}

The rationale of Definition \ref{eq:directionnal_smoothness} is to replace the norm in the usual smoothness condition (Definition \ref{def:smoothness_fun}) by a scalar product between the \textit{direction} and the negative gradient, in order to get an affine invariant quantity for the FW direction (see Proposition \ref{prop:affine_invariance_directional_smoothness} below).

Assuming $\delta(x)$ is a descent direction, i.e., $\langle - \nabla f(x),\, \delta(x)\rangle > 0$, we can obtain a minimization algorithm for $f$, by minimizing \eqref{eq:directionnal_smoothness} over $h$,
\[
    x_{k+1} = x_{k} + h_{\text{opt}} \delta(x_k), \;\; h_{\text{opt}} = \min\{h_{\max}\;;\;\mathcal{L}^{-1}_{f,{\delta}} \}.
\]

\begin{example} \textit{(Gradient descent on smooth functions)}
The gradient algorithm uses $\delta(x) = -\nabla f(x)$. In such case, the function is directionally smooth with constant $L$, and we obtain
\begin{align*}
    f(x_{k+1}) & \leq f(x_k) - h \|\nabla f(x) \|^2 + \textstyle \frac{L h^2}{2}\|\nabla f(x)\|^2 \\
    & = f(x) - h\left( \textstyle \frac{Lh}{2}-1\right)\|\nabla f(x)\|^2.
\end{align*}
The best $h$ is given by $h_{\text{opt}} = \frac{1}{L}$, which is also the optimal one  \citep{nesterov2013introductory}.
\end{example}

The advantage of directional smoothness is its affine invariance in the case where $\delta(x)$ is the FW step.
\begin{proposition}[Affine Invariance of $\mathcal{L}_{f,{\delta}}$]\label{prop:affine_invariance_directional_smoothness}
If $\delta(x)$ is affine covariant (e.g.\ the FW direction $\delta(x) \triangleq v(x)-x$), then $\mathcal{L}_{f,{\delta}}$ in \eqref{eq:directionnal_smoothness} is invariant to an affine transformation of the constraint set  (proof in Appendix \ref{sec:affine_invariance_directional_smoothness}).
\end{proposition}

The next theorem shows that, in the case of the FW algorithm, the directional smoothness constant is bounded if the function is smooth and the set is strongly convex for any distance function $\omega$. We use this result later, to show that affine invariant backtracking line-search is equivalent to using the best distance function $\omega$ to define $L_{\omega},\,c_{\omega}$ and $\alpha_\omega$.
\begin{theorem}[Directional Smoothness of FW] \label{thm:directionnal_smoothness_bound}
    Consider the function $f$, smooth w.r.t.\ the distance function $\omega$, with constant $L_\omega$, and the set $\mathcal{C}$, strongly convex with constant $\alpha_\omega$.\\ Let $\delta(x) = x-v(x)$, $v(x)$ being the FW corner
    \[
        v(x) \defas \argmin_{v\in \mathcal{C}} \langle \nabla f(x),\, v \rangle.
    \]
    Then, if $\omega_*(-\nabla f(x)) > c_{\omega}$ for all $x\in\mathcal{C}$, the function $f(x)$ is directionally smooth w.r.t.\ to $\omega$, with constant
    \begin{equation} \label{eq:bound_ratio_directional_smoothness}
        \mathcal{L}_{f,{\delta}} \leq \frac{L_{\omega}}{c_{\omega}\alpha_{\omega}}.
    \end{equation}
\end{theorem}
\begin{proof}
    See Appendix \ref{sec:directionnal_smoothness_bound} for the proof.
\end{proof}

\section{Affine Invariant Linear Rates}\label{sec:affine_invariant_analysis}

With the directional smoothness constant $\mathcal{L}_{f.\delta}$ (affine invariant when $\delta$ is the FW direction), Theorem \ref{thm:affine_invariant_linear_convergence} shows an affine invariant linear rate of convergence of FW, generalizing existing convergence results of Frank-Wolfe on strongly convex sets \citep{levitin1966constrained,demyanov1970,dunn1979rates}.

\begin{theorem}[Affine Invariant Linear Rates]\label{thm:affine_invariant_linear_convergence}
    Assume $f$ is a convex function and directionally smooth with direction function $\delta$ with constant $\mathcal{L}_{f,{\delta}}$. 
    Then, the FW Algorithm \ref{algo:FW_general} with step size
    \[
        \textstyle h_{\text{opt}} = \min\left\{1, \; \frac{1}{\mathcal{L}_{f,{\delta}}} \right\}, \quad \text{with } \delta = v(x)-x,
    \]
    or with line-search, where $v(x)$ is the FW corner
    \[
        v(x) = \argmin_{v\in\mathcal{C}} \langle \nabla f(x),\, v \rangle,
    \]
    converges linearly, at rate
    \[
        \textstyle f(x_k)-f_\star \leq  \max\left\{ \frac{1}{2},\; 1- \frac{1}{2\mathcal{L}_{f,{\delta}}} \right\}\left(f(x_{k-1})-f_\star\right).
    \]
\end{theorem}
\begin{proof}
    We start with the directional smoothness assumption. For $ 0<h<1$,
    \begin{align*}
        f\big(x_{k+1}\big) \leq & \textstyle  f(x_k) + \left(h- \frac{\mathcal{L}_{f,{\delta}} h^2}{2}\right)\langle \nabla f(x_k),\, \delta(x_k) \rangle 
    \end{align*}
    After minimization, we have two possibilities: $h_{\text{opt}} = \frac{1}{\mathcal{L}_{f,{\delta}}}$ or $h_{\text{opt}} = 1$. In the first case, we obtain
    \[
         \textstyle f\big(x_{k+1}\big) \leq f(x_k) + \frac{1}{2\mathcal{L}_{f,{\delta}}} \langle \nabla f(x_k),\, \delta(x_k) \rangle 
    \]
    Notice that the scalar product in the right-hand-side is the negative dual gap of Frank-Wolfe, that satisfies
    \[
        \langle\nabla f(x_k),\, v(x)-x \rangle \leq -\left( f(x_k)-f_{\star}  \right),
    \]
    which gives the desired result. The second case follows immediately.
\end{proof}


This provides an affine invariant analysis of the linear convergence regimes of FW on strongly convex sets. 

The next proposition shows that the directional constant in Theorem \ref{thm:affine_invariant_linear_convergence} is bounded by \eqref{eq:bound_ratio_directional_smoothness} w.r.t.\ the distance function $\omega$ that gives the best ratio. This means that the Frank-Wolfe method acts like it optimizes the function in the best possible geometry, i.e., the geometry that gives the \textit{best constants}.

\begin{proposition}[Optimality of Dir. Smoothness] 
    Let $\Omega$ the set of function defined as
    \[
        \Omega = \{\omega : \omega \text{ satisfies assumptions \ref{assum:distance_fun}}\}.
    \]
    Then, the directional smoothness constant follows
    \[
        \mathcal{L}_{f,{\delta}} \leq \min_{\omega \in \Omega} \frac{L_\omega}{c_\omega\alpha_\omega},
    \]
    where $L_\omega$ is the smoothness constant of the function $f$, $\alpha_\omega$ the strong convexity of the set $\mathcal{C}$ and \[
    c_\omega\leq \omega_*\big(-\nabla f(x)\big), \quad \forall x\in\mathcal{C}.\]
\end{proposition}
\begin{proof}
    The proof is immediate by noticing that the FW algorithm do not use $\omega$, therefore we can choose the best $\omega$ in Theorem \ref{thm:directionnal_smoothness_bound}.
\end{proof}

To obtain a similar affine invariant analysis without restriction on the position of the optimum, \textit{i.e.} the $\mathcal{O}(1/K^2)$ analysis in \citep{garber2015faster}, one can define a similar property to the direction smoothness defined in Section \ref{sec:directional_smoothness}.
This new structural assumption additionally mingles together with the strong convexity of $f$. We provide details in Appendix \ref{app:affine_invariant_garber_analysis}.
We choose to focus the analysis for the linear convergence in the main text as it is the one most significant in practice.

\section{Affine Invariant Backtracking}\label{sec:practical_affine_invariant_backtracking}

In previous sections, we proposed new constants to bound the rate of convergence of the Frank-Wolfe algorithm, which is affine invariant. The significant advantage of these constants is that, like FW, they are independent of any norm. However, the optimal step size of Frank-Wolfe needs the knowledge of these constants.  

We propose in this section an affine invariant backtracking technique (Algorithm \ref{algo:backtraking_directionnal_smooth}), based on directional smoothness. By construction, the backtracking technique finds automatically an estimate of the directional smoothness that satisfies
\[
    \mathcal{L}_k < 2\mathcal{L}_{f,\delta}, \quad k \geq \textstyle \log_2\left(\frac{\mathcal{L}_0}{\mathcal{L}_{f,\delta}}\right).
\]

\begin{algorithm}
  \caption{Affine invariant backtracking
  }\label{algo:backtraking_directionnal_smooth}
  \begin{algorithmic}[1]
    \REQUIRE FW corner $v_k$, point $x_k$, directional smoothness estimate $\mathcal{L}_k$, function $f$.
    \STATE Define the optimal step size and next iterate in the function of the directional Lipchitz constant: 
    \begin{align*}
        \gamma_\star(\mathcal{L}) & \defas \min\{\textstyle \frac{1}{\mathcal{L}},1\},\\
        x(\mathcal{L}) & \defas (1-\gamma_\star(\mathcal{L}))x_k + \gamma_{\star}(\mathcal{L})v_k.
    \end{align*}
    \STATE Create the model of $f$ between $x_k$ and $x(\mathcal{L})$ based on equation \eqref{eq:directionnal_smoothness},
    \[
        m(\mathcal{L}) \defas f(x_k) + \gamma_\star(\mathcal{L})\left(1-\gamma_\star(\mathcal{L})\right) \langle \nabla f(x_k),\, v_k-x_k\rangle
    \]
    \STATE Set the current estimate $\tilde{\mathcal{L}} \defas \frac{\mathcal{L}_k}{2}$.
    \WHILE{ $f( x(\tilde{\mathcal{L}}) ) > m(\tilde{\mathcal{L}})$ (Sufficient decrease not met because $\tilde{\mathcal{L}}$ is too small)}
        \STATE Double the estimate : $\tilde{\mathcal{L}}\leftarrow 2\cdot \tilde{\mathcal{L}}$.
    \ENDWHILE
    \ENSURE Estimate $\mathcal{L}_{k+1} = \tilde{\mathcal{L}}$, iterate $x_{k+1} = x(\tilde{\mathcal{L}})$
  \end{algorithmic}
\end{algorithm}

\section{Why Backtracking FW with norms is so efficient?}\label{sec:explanation_efficiency_ls}

The step size strategy in Frank-Wolfe usually drives its practical efficiency. Sometimes, setting the step size optimally w.r.t.\ the theoretical analysis may be suboptimal in practice. Recently, \citet{pedregosa2020linearly} analyze the rate of the Frank-Wolfe algorithm for smooth function, using \textit{backtracking line search}, described in Algorithm \ref{algo:backtraking_smooth}, Appendix \ref{sec:backtraking_smooth}.

Algorithm \ref{algo:backtraking_smooth} in Appendix \ref{sec:backtraking_smooth} is adaptive to the local smoothness constant, and ensures $L_{k+1} < 2 L_f$, $L_f$ being the smoothness constant of the function in the $\ell_2$ norm. \citet{pedregosa2020linearly} observed that the estimate of the Lipchitz constant is often significantly smaller than the theoretical one; they wrote: \textit{``We compared the average Lipschitz estimate $L_t$ and the $L$, the gradient’s Lipschitz constant. We found that across all datasets the former was more than an order of magnitude smaller, highlighting the need to use a local estimate of the Lipschitz constant to use a large step size."}

With our analysis, however, we can explain why the estimate of the smoothness constant is much better than the theoretical one. The answer is simple:
\begin{center}
    \textit{Despite using a non-affine invariant bound, the step size resulting from the estimation of the Lipchitz constant via the backtracking line-search finds $\frac{1}{\mathcal{L}_{f,{\delta}}}$.}
\end{center}
\begin{proposition} \label{prop:aff_inv_ls}
    Consider the ``local Lipchitz constant'' $L_{\text{loc}}(x)$ that satisfies \eqref{eq:smoothness_fun} with $y = x+h\delta(x)$, i.e.,
    \begin{align*}
        f(x+h\delta(x)) \leq & f(x) + \nabla f(x)(x+h\delta(x)) \\& + \textstyle L_\text{loc}(x) \frac{h^2}{2}\|\delta(x)\|^2_2.
    \end{align*}
    Then, $L_{\text{loc}}(x)$ is bounded by
    \[
        L_{\text{loc}}(x) \leq  \mathcal{L}_{f,{\delta}}\frac{\langle - \nabla f(x),\delta(x)\rangle}{\|\delta(x)\|^2}.
    \]
    Assuming $L_{\text{loc}}(x)$ ``locally constant", the backtracking line-search finds $L_k < 2 L_{\text{loc}}(x_k)$, and its step size $\gamma_{\star}$ satisfies
    \[
        \min\left\{1,\frac{1}{2\mathcal{L}_{f,{\delta}}}\right\} \leq \gamma_\star.
    \]
\end{proposition}
\begin{proof}
    See Appendix \ref{sec:aff_inv_ls} for the proof.
\end{proof}

Therefore, the optimal step size from the backtracking line-search with the $\ell_2$ norm is \textit{exactly} the optimal affine invariant step size of our affine invariant analysis from Theorem \ref{thm:affine_invariant_linear_convergence}. 

In conclusion, \textit{even if we use non-affine invariant norms} to find the smoothness constant, surprisingly, \textit{the backtracking procedure finds the optimal, affine invariant step size}.

\section{Illustrative Experiments}\label{sec:experiments}

\begin{figure*}[t!]
    \centering
    \includegraphics[width=0.45\linewidth]{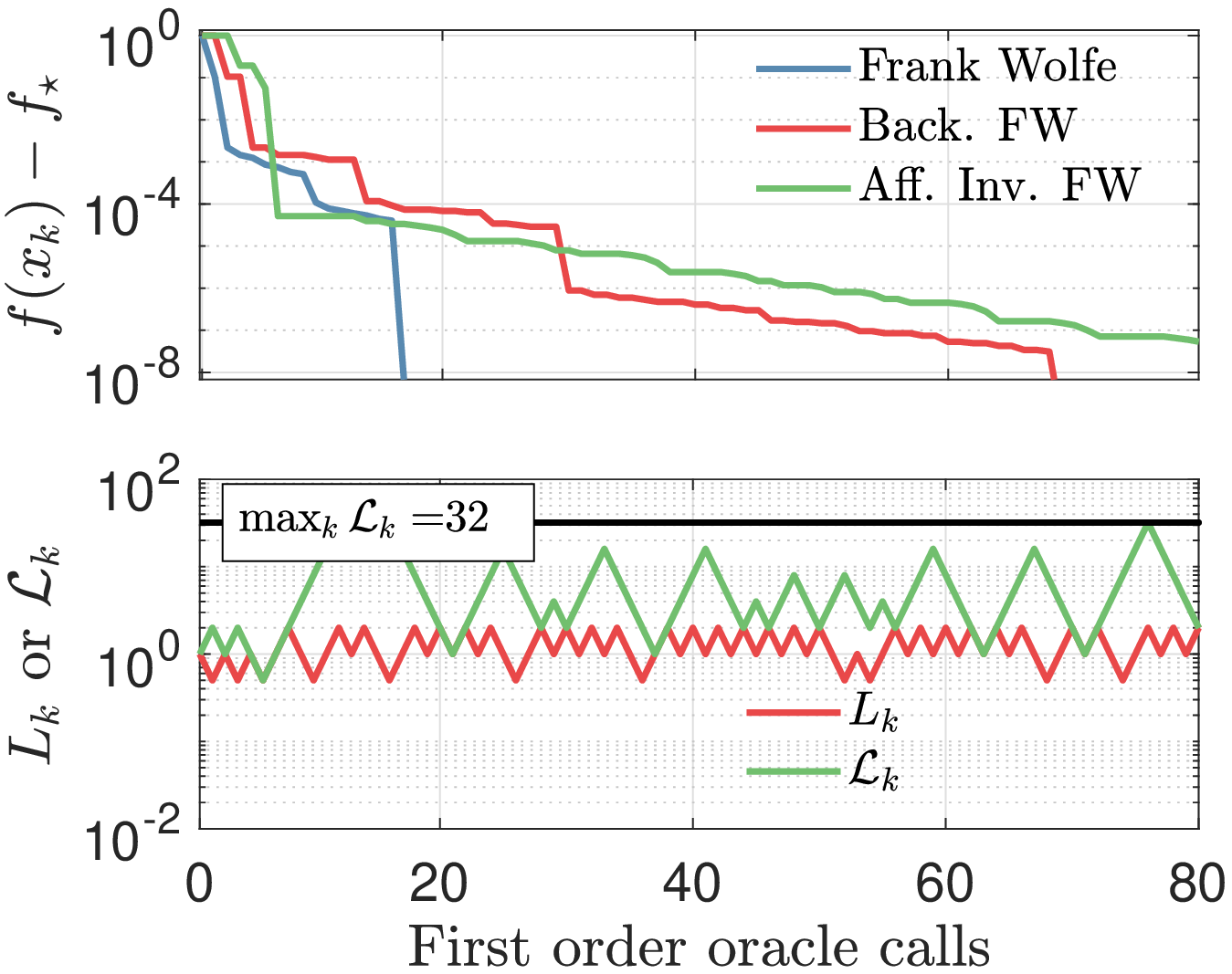}
    \includegraphics[width=0.45\linewidth]{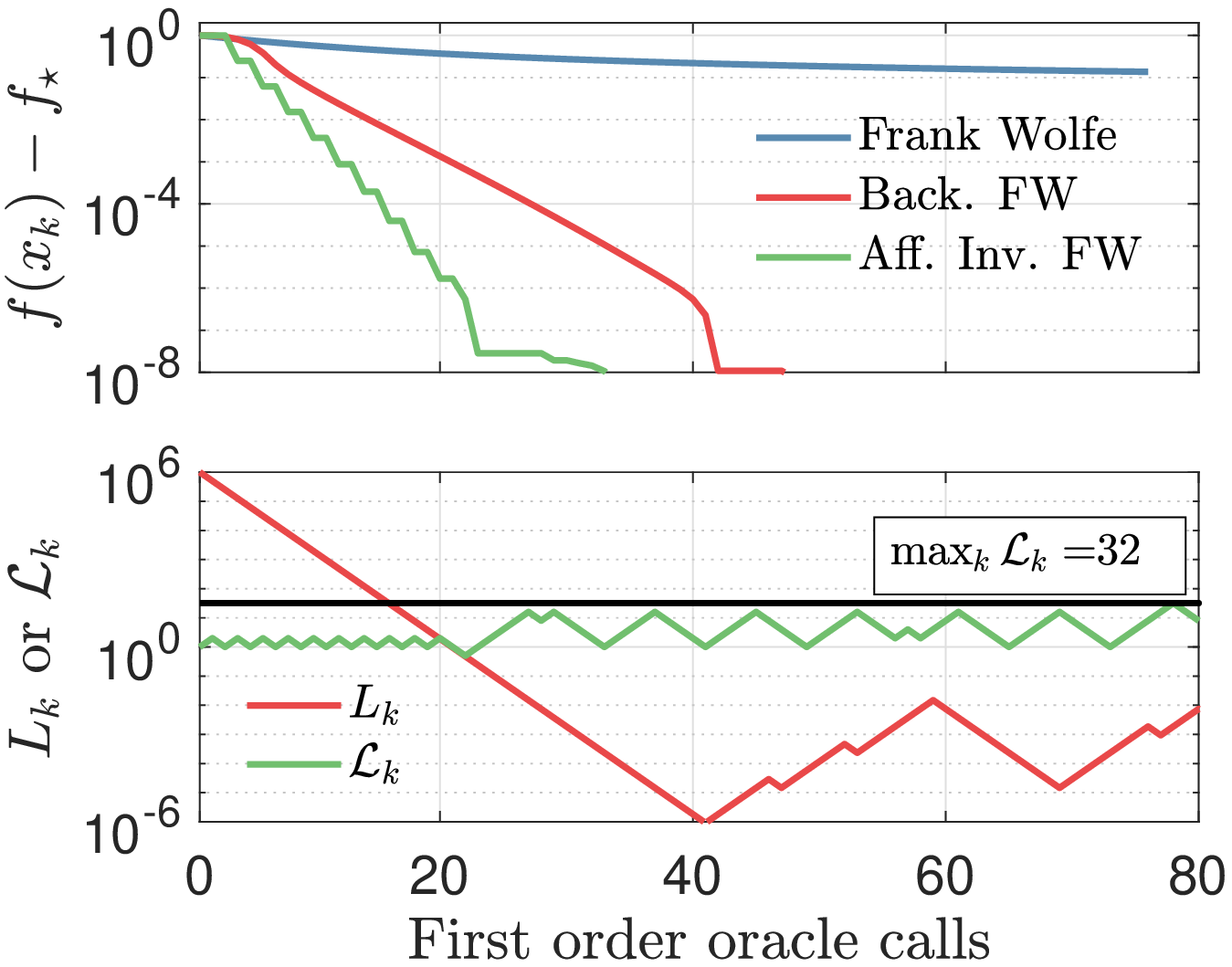}
    \caption{Comparison of FW variants on the projection problem. Left: $B=I$, Right: $\kappa(B) = 10^6$. The top row is the gap $f_k-f^*$, and the bottom row corresponds to the estimation of the directional-smoothness constant $\mathcal{L}_k$ or the smoothness constant $L_k$, where the black line report the maximum value of $\mathcal{L}_k$. The reason why adaptive FW methods are slower in the left figure is because, in the worst case, the number of iterations to reach a certain precision can be up to four times larger than the worst-case bound on non-adaptive methods. We clearly see that the directional smoothness parameter $\mathcal{L}_{f,\delta}$ is affine invariant, as its estimate is $\max_k\mathcal{L}_k = 32$ in both scenarios.}    \label{fig:projection}\end{figure*}

\begin{figure}[t]
    \centering
     \includegraphics[width=0.9\linewidth]{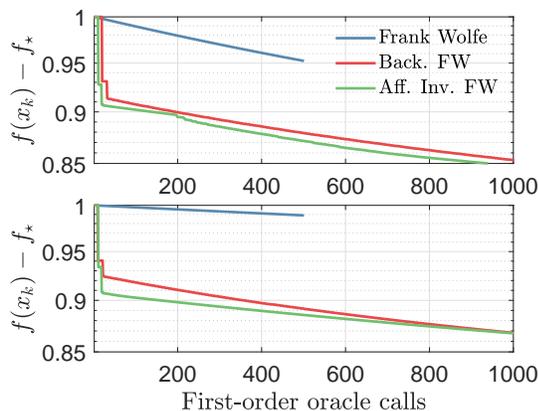}
    \caption{Classification problem on Madelon dataset, with (\textit{Top}) Quadratic loss and (\textit{Bottom}) Logistic loss.}
    \label{fig:logistc}
\end{figure}

\paragraph{Quadratic / logistic regression.} We consider the constrained quadratic and logistic regression problem,
\begin{align}
    \min_{x\in\cC}\frac{1}{n}\sum_{i=1}^n l(a_i^T x, y_i), \quad 
 \end{align}
where $l$ is the quadratic or the logistic loss. Here we adopt the $\ell_2$-ball, defined as
\[
    \cC = \{x: \|x\|_2\leq R\} , \quad R>0.
\]
Specifically, we compare our affine invariant backtracking method in Algorithm~\ref{algo:backtraking_directionnal_smooth} against the naive FW Algorithm~\ref{algo:FW_general} with step size $1/L$ \citep{demyanov1970} and back-tracking FW~\citep{pedregosa2020linearly} on the Madelon dataset~\citep{guyon2007competitive}. The results are shown in  Figure~\ref{fig:logistc}.
In detail, we set $R$ such that the unconstrained optimum $\xx^*$  satisfies  $\|\xx^*\|_2 = 1.1 R$, and the initial iterate $\xx_0= \bf{0}$. As predicted by our theory, the affine invariant algorithm performs well at the beginning, but after a few iterations the two backtracking techniques behave similarly.

\paragraph{Projection.} We solve here the projection problem described in Example \ref{eq:example_not_affine_invariant}, for two cases of $B$: One that corresponds to the original problem, i.e. $B=I$, the second one where $B$ is an ill-conditioned matrix (with the condition number $\kappa(B)=10^6$). The vector $x_0$ is random in the $\ell_2$ ball, and $\bar{x} = \textbf{1}_d\cdot (1.1/\sqrt{d})$. We report the results in Figure \ref{fig:projection}. We compare the standard FW algorithm with step size $1 / L$, the FW with backtracking line-search (Algorithm \ref{algo:backtraking_smooth}) and FW with affine invariant backtracking technique (Algorithm \ref{algo:backtraking_directionnal_smooth}). If the problem is well-conditioned ($\kappa(B) = 1$), all methods perform similarly. This is not the case, however, for the ill-conditioned setting, where the FW with no adaptive step size converges extremely slowly compared to the two other methods. We also see that the affine invariant backtracking converges quicker than the standard backtracking. This is explained by the fact that the latter takes a longer time to find the right constant $L_k$, while $\mathcal{L}_k$ remains untouched after an affine transformation.

\section{Conclusion}

In this paper, our theoretical convergence results on strongly convex sets complete the series of accelerated affine invariant analyses of Frank-Wolfe algorithms. To obtain these, we formulate a new structural assumption with respect to general distance functions, the directional smoothness, which we will explore more systematically in future works.  Also, we present a new affine invariant backtracking line-search method based on directional smoothness. Within our framework of analysis, we provide a new explanation for the reasons behind the efficiency of the existing backtracking line search, and we show theoretically and experimentally they also find affine-invariant step sizes.

\clearpage

\section*{Acknowledgments}

This research was partially supported by the Canada CIFAR AI Chair Program. Simon Lacoste-Julien is a CIFAR Fellow in the Learning in Machines \& Brains program.

%

\printbibliography

\newpage
\appendix

\onecolumn

\section{Strong Convexity of Sets with asymmetric distance functions}\label{app:proofs_assymetric_gauges}

Before presenting the proof, we introduce the following results, extending known properties from smooth and strongly convex sets.
\begin{proposition}
    If $f$ is strongly convex w.r.t.\ the distance function $\omega$, then for $\gamma \in [0,1]$ we have
    \[
        f(\gamma x + (1-\gamma)y) + \mu \gamma(1-\gamma) \frac{\gamma \omega^2(x-y) + (1-\gamma)\omega^2(y-x)}{2} \leq \gamma f(x) + (1-\gamma) f(y)
    \]
\end{proposition}
\begin{proof}
    Let $z_\gamma = \gamma x + (1-\gamma)y$. We start with the definition,
    \begin{align*}
        f(z_{\gamma}) + \langle \nabla f(z_{\gamma}), \, x-z_{\gamma} \rangle + \frac{\mu}{2} \omega^2(x-z_{\gamma}) & \leq f(x) \\
        f(z_{\gamma}) + \langle \nabla f(z_{\gamma}), \, y-z_{\gamma} \rangle + \frac{\mu}{2} \omega^2(y-z_{\gamma}) & \leq f(y)
    \end{align*}
    After multiplying by $\gamma$ and $1-\gamma$ and adding the two inequalities, we have
    \[
        f(z_{\gamma}) + \mu\frac{ \gamma \omega^2(x-z_{\gamma}) + (1-\gamma)\omega^2(y-z_{\gamma})}{2} \leq \gamma f(x) + (1-\gamma) f(y)
    \]
    Since $\omega^2(x-z_{\gamma}) = (1-\gamma)^2 \omega^2(y-x)$, and $\omega^2(y-z_{\gamma}) = \gamma^2 \omega^2(x-y)$, we obtain the desired result.
\end{proof}

\begin{proposition}\label{prop:dual_smooth}
    If $f$ is convex and smooth w.r.t.\ the distance function $\omega$, then it holds that
    \[
         \frac{1}{2L} \omega_*^2\big( \nabla f(x)-\nabla f(y) \big) \leq f(y)-f(x) - \langle \nabla f(x), \, y-x \rangle 
    \]
    where $\omega_*$ is the dual of the function $\omega$, written
    \[
        \omega_*(v) \defas \max_{s:\omega(s) \leq 1} \langle v,\,s \rangle.
    \]
\end{proposition}
In particular, Proposition \ref{prop:dual_smooth} implies that, if $f$ has a minimum $x_\star$, then
\begin{equation}
    \frac{1}{2L} \omega_*^2\big( -\nabla f(y) \big) \leq f(y)-f(x_\star)
\end{equation}
\begin{proof}
    Let the function $\phi(y) = f(y)-\langle \nabla f(x), y\rangle $. This function is, by construction, smooth. Moreover, $\min_y \phi(y)$ is attained when $y=x$. Since the function is smooth,
    \[
        \min_{y} \phi(y) \leq \min_y \phi(z) +  \langle \nabla \phi(z),\, y-z \rangle + \frac{L}{2} \omega^2(y-z)
    \]
    Let $ \beta u = y-z$, where $\omega(u) = 1$ and $\beta \geq 0$. Then,
    \[
        \min_{y} \phi(y) \leq \min_{\beta, u} \phi(z) +  \beta \langle \nabla \phi(z),\, u \rangle + \frac{\beta^2 L}{2}
    \]
    The minimum can be split into two minimization problems,
    \[
        \min_{y} \phi(y) \leq \phi(z) +  \min_{\beta\geq 0}\left( \frac{\beta^2 L}{2} - \beta  \max_{u : \omega(u) \leq 1} \langle - \nabla \phi(z),\, u \rangle\right).
    \]
    By definition of the dual of $\omega$,
    \[
        \min_{y} \phi(y) \leq \phi(z) +  \min_{\beta\geq 0}\left( \frac{\beta^2 L}{2} - \beta  \omega_{*} \big( - \nabla \phi(z)\big)\right).
    \]
    Now, we can solve over $\beta$, which gives us
    \[
        \min_{y} \phi(y) \leq \phi(z) - \frac{1}{2L} \omega_{*}^2 \big( - \nabla \phi(z)\big).
    \]
    Replacing the minimum by $\phi(x)$, and $\phi$ by its expression, we get
    \[
        f(x)- \langle \nabla f(x),\, x \rangle \leq  f(z)- \langle \nabla f(x),\, z \rangle - \frac{1}{2L} \omega_{*}^2 \big( \nabla f(x) - \nabla f(z)\big).
    \]
    After reorganization, we get the desired result.
\end{proof}

We can now show that level sets of a smooth and strong convex function are strongly convex sets, when they use the distance function $\omega$.

\begin{proof}
    (\textbf{Proof of Lemma~\ref{thm:level_set_minkow_fun}.}) Note to the reviewers: there was a small typo in our proof that was caught after the main paper deadline: the correct constant is actually $\alpha_\omega= \frac{ \mu }{\kappa_{\omega} \sqrt{2LR}}$ (i.e.\ the asymmetry factor $\kappa_\omega$ does appear in the expression, unlike was originally mentioned in the main text of Lemma~\ref{thm:level_set_minkow_fun}). This change is minor and does not change the rest of the story of the paper. 
    
    Consider the set
    \[
        \mathcal{C} = \{ x : f(x)-f_{\star} \leq R \}
    \]
    Let $x,\, y \in \mathcal{C}$.  Let $z_\gamma = \gamma x + (1-\gamma)y$, and consider the point $z_\gamma + u$. We have that
    \begin{align*}
        f(z_{\gamma} + u)-f_{\star} & \leq f(z_{\gamma})-f_{\star} + \langle \nabla f(z_{\gamma}),\, u \rangle + \frac{L}{2}\omega^2(u), \\
        & \leq f(z_{\gamma})-f_{\star} + \omega(-u)\max_{v : \omega(v) \leq 1}\langle -\nabla f(z_{\gamma}),\, v \rangle + \frac{L}{2}\omega^2(u),\\
        & = f(z_{\gamma})-f_{\star} + \omega(-u) \omega_*\big(-\nabla f(z_{\gamma})\big) + \frac{L}{2}\omega^2(u),\\
        & \leq f(z_{\gamma})-f_{\star} + \kappa_{\omega} \omega(u) \sqrt{ 2L(f(z_{\gamma})-f_{\star}) }  + \frac{L}{2}\omega^2(u).
    \end{align*}
    Therefore, to satisfy $f(z_{\gamma} + u)-f_{\star} \leq R$, we need to ensure that
    \[
        \underbrace{f(z_{\gamma})-f_{\star} - R}_{=\omega} + \underbrace{\kappa_{\omega}  \sqrt{ 2L(f(z_{\gamma})-f_{\star}) }}_{=\beta} \omega(u)  + \frac{L}{2}\omega^2(u) \leq 0
    \]
    Solving the problem in $\omega(u)$ gives
    \[
        \omega(u) \leq \frac{-\beta + \sqrt{\beta^2 - 2L\omega }}{L}
    \]
    We have that
    \[
        \beta^2 - 2L\omega =  2L \left( (f(z_{\gamma})-f_{\star}) (\kappa_{\omega}^2 -1) + R \right)
    \]
    Therefore,
    \[
        \omega(u) \leq \sqrt{2}\frac{-\kappa_{\omega}  \sqrt{ (f(z_{\gamma})-f_{\star}) } + \sqrt{(f(z_{\gamma})-f_{\star}) (\kappa_{\omega}^2 -1) + R }}{\sqrt{L}}
    \]
    However, since the function is strongly convex,
    \[
        f(z_{\gamma})-f_\star \leq \underbrace{\gamma f(x) + (1-\gamma) f(y)-f_\star}_{\leq R} - \mu \gamma(1-\gamma) \frac{\gamma \omega^2(x-y) + (1-\gamma)\omega^2(y-x)}{2}
    \]
    Let $D_\gamma = \gamma(1-\gamma) \frac{\gamma \omega^2(x-y) + (1-\gamma)\omega^2(y-x)}{2}$. The inequality now reads
    \begin{equation}
        f(z_{\gamma})-f_\star \leq R - \mu D_{\gamma}. \label{eq:bound_R_Dgamma}
    \end{equation}
    Therefore, the condition on $\omega$ becomes
    \[
        \omega(u) \leq \sqrt{2}\frac{-\kappa_{\omega}  \sqrt{ R-\mu D_\gamma } + \sqrt{(R-\mu D_\gamma) (\kappa_{\omega}^2 -1) + R }}{\sqrt{L}}
    \]
    which gives
    \begin{equation} \label{eq:omega_u_condition_complicated}
        \omega(u) \leq \frac{\kappa_{\omega}\sqrt{2}}{\sqrt{L}} \left( -\sqrt{ R-\mu D_\gamma } + \sqrt{R - \left(1-\frac{1}{\kappa_{\omega}^2}\right)\mu D_\gamma}\right)
    \end{equation}
    
    To simplify the expression in parenthesis, we multiply and divide by the conjugate of the square roots to get:
    
    \begin{align*}
    \left( -\sqrt{ R-\mu D_\gamma } + \sqrt{R - \left(1-\frac{1}{\kappa_{\omega}^2}\right)\mu D_\gamma}\right) &=
        \frac{ R - \left(1-\frac{1}{\kappa_{\omega}^2}\right)\mu D_\gamma - ( R-\mu D_\gamma)}
        {\sqrt{ R-\mu D_\gamma } + \sqrt{R - \left(1-\frac{1}{\kappa_{\omega}^2}\right)\mu D_\gamma} } \\
        &\geq \frac{1}{\kappa_{\omega}^2 2 \sqrt{R}} .
    \end{align*}
    
    We can thus strengthen the condition~\eqref{eq:omega_u_condition_complicated} to:
    \[
    \omega(u) \leq \frac{ \mu D_\gamma }{\kappa_{\omega} \sqrt{2LR}}.
    \]
    
    As the definition of a strongly convex set requires $\omega(u) \leq \alpha_\omega D_\gamma$, we conclude that the level set is strongly convex with at least the constant $\alpha_\omega= \frac{ \mu }{\kappa_{\omega} \sqrt{2LR}}$.
    \remove{
    Technically, we can use the condition as it, but in order to simplify, we use the concavity of $\sqrt{\cdot}$ as follow,
    \[
        \sqrt{y} \leq  \sqrt{x} -\frac{1}{2\sqrt{x}}(y-x) \quad \Rightarrow \quad \frac{1}{2\sqrt{x}}(y-x)\leq  \sqrt{x} - \sqrt{y}.
    \]
    With $x = R - \left(1-\frac{1}{\kappa_{\omega}^2}\right)\mu D_\gamma$ and $y =  R-\mu D_\gamma$, we obtain the sufficient condition
    \begin{align}
        \omega(u) & \leq  \frac{\kappa_{\omega}\sqrt{2}}{\sqrt{L}} \frac{\frac{\mu D_\gamma}{\kappa_{\omega}^2}}{2\sqrt{R - \left(1-\frac{1}{\kappa_{\omega}^2}\right)\mu D_\gamma}} \nonumber\\
        & = \frac{1}{\kappa_{\omega}\sqrt{2L}} \frac{\mu D_\gamma}{\sqrt{R - \left(1-\frac{1}{\kappa_{\omega}^2}\right)\mu D_\gamma}} \label{eq:temp_cond_d}
    \end{align}
    
    Since $\mu D_\gamma \leq R$ by equation \eqref{eq:bound_R_Dgamma}, we have
    \begin{align*}
        \frac{\mu D_\gamma}{\sqrt{R - \left(1-\frac{1}{\kappa_{\omega}^2}\right)\mu D_\gamma}}  & {\color{red} \geq }\frac{\mu D_\gamma}{\sqrt{R}\sqrt{1 - \left(1-\frac{1}{\kappa_{\omega}^2}\right)}}\\
        & = \frac{\mu D_\gamma \kappa_{\omega}}{\sqrt{R}}
    \end{align*}
    Therefore, the condition \eqref{eq:temp_cond_d} can be strengthen into
    \[
        \omega(u) \leq \frac{ \mu D_\gamma }{\sqrt{2LR}}.
    \]
    This means the level set is strongly convex with constant $\alpha = \frac{ \mu }{\sqrt{2LR}}$.
    }
\end{proof}

\subsection{Proof of Theorem \ref{thm:directionnal_smoothness_bound}} \label{sec:directionnal_smoothness_bound}
\begin{theorem} 
    Consider the function $f$, smooth w.r.t.\ the distance function $\omega$, with constant $L_\omega$, and the set $\mathcal{C}$, strongly convex with constant $\alpha_\omega$.\\ Let $\delta(x) = x-v(x)$, $v(x)$ being the FW corner
    \[
        v(x) \defas \argmin_{v\in \mathcal{C}} \langle \nabla f(x),\, v \rangle.
    \]
    Then, if $\omega_*(-\nabla f(x)) > c_{\omega}$ for all $x\in\mathcal{C}$, the function $f(x)$ is directionally smooth w.r.t.\ to $\omega$, with constant
    \begin{equation}
        \mathcal{L}_{f,{\delta}} \leq \frac{L_{\omega}}{c_{\omega}\alpha_{\omega}}.
    \end{equation}
\end{theorem}
\begin{proof}
    We start by the definition of smooth functions between $x$ and $h \delta(x)$ for the distance function $\omega$. We have for all $0\leq h \leq 1$
    \[
        f(x + h \delta(x)) \leq f(x) + h\langle \nabla f(x),\, \delta(x) \rangle + \frac{h^2 L_{\omega}}{2} \omega^2(\delta(x))
    \]
    Using the scaling inequality in \eqref{eq:scaling_inequality},
    \[
        \langle -\nabla f(x) , \, \delta(x) \rangle \geq \alpha_\omega \omega_*\big(-\nabla f(x)\big) \omega(\delta(x))^2.
    \]
    We hence obtain
    \begin{align*}
        f(x+h\delta(x)) \leq & f(x) + h\langle \nabla f(x),\, \delta(x) \rangle - \frac{h^2 L_{\omega}}{2} \frac{\langle \nabla f(x) , \, \delta(x) \rangle}{\alpha_{\omega}  \omega_*\big(-\nabla f(x)\big)}.
    \end{align*}
    Since $\omega_*(-\nabla f(x)) > c_{\omega}$ for all $x\in\mathcal{C}$,
    \begin{align*}
        f(x+h\delta(x)) \leq & f(x) + h\langle \nabla f(x),\, \delta(x) \rangle - \frac{h^2}{2} \frac{L_{\omega}}{\alpha_{\omega}  c_{\omega}} \langle \nabla f(x) , \, \delta(x) \rangle.
    \end{align*}
    which is the definition of directional smoothness.
\end{proof}

\section{Missing proofs}

\subsection{Proof of Proposition \ref{prop:aff_inv_ls}} \label{sec:aff_inv_ls}

\begin{proposition} 
    We define the ``local Lipchitz constant'' $L_{\text{loc}}(x)$, which satisfies
    \[
        L_{\text{loc}}(x) \defas  \mathcal{L}_{f,{\delta}}\frac{\langle - \nabla f(x),\delta(x)\rangle}{\|\delta(x)\|^2}.
    \]
    Then, assuming that the local Lipchitz constant is ``locally constant", the backtracking line-search finds $L_k \leq 2 L_{\text{loc}}(x_k)$, and its step size $\gamma_{\star}$ satisfies
    \[
        \min\left\{1,\frac{1}{2\mathcal{L}_{f,{\delta}}}\right\} \leq \gamma_\star.
    \]
\end{proposition}
\begin{proof}
We start with the definition of directional smoothness,
\begin{align*}
f(x+h\delta(x)) \leq & f(x) + h \langle\nabla f(x),\,\delta(x)\rangle  + \left[\mathcal{L}_{f,{\delta}}\langle - \nabla f(x),\delta(x)\rangle\right] \frac{h^2}{2}.
\end{align*}
Writing $1 = \frac{\|\delta(x)\|_2^2}{\|\delta(x)\|_2^2}$, the upper bound becomes
\begin{align*}
    & f(x) + h \langle\nabla f(x),\,\delta(x)\rangle  \quad + \left[\frac{\mathcal{L}_{f,{\delta}}\langle 
    - \nabla f(x),\delta(x)\rangle}{\|\delta(x)\|_2^2}\right] \frac{h^2\|\delta(x)\|_2^2}{2}.
\end{align*}
Defining 
\[
    L_{\text{loc}}(x)\triangleq\frac{\mathcal{L}_{f,{\delta}}\langle - \nabla f(x),\delta(x)\rangle}{\|\delta(x)\|_2^2},
\]
we obtain
\begin{align*}
    f(x_k+h\delta(x_k)) \leq & f(x_k) + h \langle\nabla f(x_k),\,\delta(x_k)\rangle  + L_{\text{loc}}(x_k) \frac{h^2\|\delta(x_k)\|_2^2}{2}.
\end{align*}
If we assume that $L_{\text{loc}}(x_k)$ is approximately constant, then Algorithm \ref{algo:backtraking_smooth} finds $L_k \leq 2L_{\text{loc}}(x_k)$. Finally, using the definition of $\gamma_{\star}$ in Algorithm \ref{algo:backtraking_smooth}, we have
\begin{align*}
    \gamma_{\star} & =  \min\left\{\frac{-\nabla f(x_k)(v_k-x_k)}{L_{\text{loc}}(x_k)\|v_k-x_k\|^2}, 1\right\}\\
    & \geq \min\left\{\frac{1}{2\mathcal{L}_{f,{\delta}}}, 1\right\}.
\end{align*}
\end{proof}

\subsection{Proof of Proposition \ref{prop:affine_invariance_directional_smoothness}} \label{sec:affine_invariance_directional_smoothness}
\begin{proposition}[Affine Invariance]
If $\delta(x)$ is affine covariant (e.g.\  the Frank-Wolfe direction $\delta(x) \triangleq v(x)-x$), then the constant $\mathcal{L}_{f,{\delta}}$ in \eqref{eq:directionnal_smoothness} is affine invariant. In other words, let 
\[
    \tilde f(\cdot) \triangleq f(B\cdot), \;\; \tilde \delta_{\tilde{\mathcal{C}}}(\cdot) \triangleq \delta_ {B\cdot\mathcal{C}}(\cdot),
\]
then $\mathcal{L}_{\tilde f,\tilde \delta_{\tilde{\mathcal{C}}}} = \mathcal{L}_{f,\delta}$.
\end{proposition}
\begin{proof}
    We start with the definition of directional smoothness, but with $x\rightarrow By$. The upper bound reads
    \begin{align*}
         f(By) &+ \left(h- \frac{\mathcal{L}_{f,\delta} h^2}{2}\right)\langle \nabla f(By),\, \delta(By) \rangle 
    \end{align*}
    Since we assumed $\delta(By)$ affine covariant,
    \[
        \delta(By) = B \tilde\delta_{\tilde{\mathcal{C}}}(y).
    \]
    Therefore,
    \begin{align*}
         f(By) &+ \left(h- \frac{\mathcal{L}_{f,\delta} h^2}{2}\right)\langle B^T\nabla f(By),\, \tilde\delta_\mathcal{\tilde C}(y) \rangle 
    \end{align*}
    Since $\nabla \tilde f(y) = B^T\nabla f(By) $, we have
    \begin{align*}
         \tilde f(\tilde y + h \tilde \delta_{\tilde{\mathcal{C}}}(y)) \leq \tilde f(y) &+ \left(h- \frac{\mathcal{L}_{f,\delta} h^2}{2}\right)\langle \nabla \tilde f(y),\, \tilde\delta_\mathcal{\tilde C}(y) \rangle 
    \end{align*}
    This means the function $\tilde f$ is directionally smooth with constant $\mathcal{L}_{f,\delta}$, which proves the statement.
\end{proof}

\section{Backtracking Line Search for Frank-Wolfe Steps} \label{sec:backtraking_smooth}

\begin{algorithm}
  \caption{Backtracking line-search for smooth functions \citep{pedregosa2020linearly}}\label{algo:backtraking_smooth}
  \begin{algorithmic}[1]
    \REQUIRE FW corner $v_k$, point $x_k$, smoothness estimate $L_k$, function $f$.
    \STATE Create the optimal step size and next iterate in the function of the Lipchitz estimate 
    \begin{align*}
        \gamma_\star(L) & \defas \min\left\{\frac{-\nabla f(x_k)(v_k-x_k)}{L\|v_k-x_k\|^2}, 1\right\}.\\
        x(L) & \defas (1-\gamma_\star(L)) + \gamma_{\star}(L)v_k
    \end{align*}
    \STATE Quadratic model of $f$ between $x_k$ and $x(L)$,
    \[
        m(L) \defas f(x_k) + \langle \nabla f(x_k),\, x(L)-x_k\rangle + \frac{L}{2}\|x(L)-x_k\|^2
    \]
    \STATE Set the current estimate $\tilde L \defas \frac{L_k}{2}$.
    \WHILE{ $f( x(\tilde L) ) > m(\tilde L)$ (Sufficient decrease not met because $\tilde L$ is too small)}
        \STATE Double the estimate : $\tilde L\leftarrow 2\cdot \tilde L$.
    \ENDWHILE
    \ENSURE Estimate $L_{k+1} = \tilde L$, iterate $x_{k+1} = x(\tilde L)$
  \end{algorithmic}
\end{algorithm}

\section{Affine Invariant Analysis without Restriction on Optimum Location}\label{app:affine_invariant_garber_analysis}
In this section, we propose a modification of the directional smoothness defined in Section \ref{sec:directional_smoothness}. This new assumption is the basis to obtain an affine invariant analysis of Frank-Wolfe on a strongly convex set without restriction on the position of the unconstrained optimum of $f$, as recently proposed in \citet{garber2015faster}.  

\paragraph{Outline.} In Theorem \ref{thm:affine_invariant_sublinear_convergence}, we prove a $\mathcal{O}(1/K^2)$ sublinear convergence rate as in \citep{garber2015faster} when the function is \textit{modified directionally smooth} (Definition \ref{def:modified_directionnal_smoothness}). In Theorem \ref{thm:modified_directionnal_smoothness_bound}, we prove that when $\mathcal{C}$ is strongly convex, and $f$ is smooth and strongly convex, then $f$ is \textit{modified directionally smooth} for the Frank-Wolfe direction with an affine invariant constant leading to better conditioned convergence rates than in \citep{garber2015faster}. Finally, in Proposition \ref{prop:affine_invariance_modified_directional_smoothness}, we show that the constant of modified directional smoothness is affine invariant.

We now define a modification of directional smoothness. It is a structural assumption on $f$ constrained on $\mathcal{C}$ designed at gathering the strong convexity of $\mathcal{C}$, the smoothness, and the strong convexity of $f$ into a single quantity.

\begin{definition}[Modified Directional Smoothness] \label{def:modified_directionnal_smoothness}
    Let $x_0\in\mathcal{C}$. The function $f$ is called \textit{modified directionally smooth} with direction function $\delta :\mathcal{C} \rightarrow \mathbb{R}^N$ if there exists a constant $\tilde{\mathcal{L}}_{f,{\delta}}(x_0)>0 $ such that $\forall x \in \mathcal{C}$,
    \begin{equation} \label{eq:modified_directionnal_smoothness}
    f\big(x + h\delta(x)\big) \leq  f(x) + h\langle \nabla f(x),\, \delta(x) \rangle - \frac{\tilde{\mathcal{L}}_{f,\delta}(x_0) h^2}{2}
    \langle \nabla f(x),\, \delta(x) \rangle
    \sqrt{\frac{f(x_0) - f^*}{f(x) - f^*}},
    \end{equation}
    for $ 0 < h < 1$.
\end{definition}

Note that the dependence of $x_0$ in the definition of the modified directional smoothness is an artifact to obtain a dimensionless constant $\tilde{\mathcal{L}}_{f,\delta}(x_0)$.

As in Section \ref{sec:affine_invariant_analysis}, the modified directional smoothness constant $\tilde{\mathcal{L}}_{f.\delta}$ is affine invariant in the case where $\delta$ is the FW direction. We now derive an affine invariant accelerated sublinear rate of convergence of Frank-Wolfe providing an affine invariant analysis of \citep{garber2015faster}.

\begin{theorem}[Affine Invariant Accelerated Sublinear Rates]\label{thm:affine_invariant_sublinear_convergence}
    Let $x_0\in\mathcal{C}$ and assume $f$ is a convex function and modified directionally smooth with direction function $\delta$ and constant $\tilde{\mathcal{L}}_{f,{\delta}}(x_0)$.
    Then, the iterates~$x_k$ for the Frank-Wolfe Algorithm~\ref{algo:FW_general} with step size
    \[
        \textstyle h_{\text{opt}} = \min\left\{1, \;  \frac{1}{\tilde{\mathcal{L}}_{f,{\delta}}(x_0)} \sqrt{\frac{f(x_k) - f^*}{f(x_0) - f^*}} \right\}, \quad \text{with } \delta = v(x)-x,
    \]
    or with exact line-search, where $v(x)$ is the Frank-Wolfe corner
    \[
        v(x) = \argmin_{v\in\mathcal{C}} \langle \nabla f(x),\, v \rangle,
    \]
    satisfy
    \[
    f(x_k) - f^* \leq \frac{4 (f(x_0) - f^*)\max\{1 , \, 18 \tilde{\mathcal{L}}^2_{f,{\delta}}(x_0) \}}{(k+2)^2} \quad \text{for $k \geq 0$.}
    \]
\end{theorem}

\begin{proof}
The proof is similar to that of Theorem \ref{thm:affine_invariant_linear_convergence}. We hence start with the modified directional smoothness assumption on $f$. For $ 0<h<1$,
    \begin{align} \label{eq:ModDirSmoothDecrease}
        f\big(x_{k+1}\big) \leq & f(x_k) + \left(h- \frac{\tilde{\mathcal{L}}_{f,{\delta}} h^2}{2} \sqrt{\frac{f(x_0) - f^*}{f(x_k) - f^*}}\right)\langle \nabla f(x_k),\, \delta(x_k) \rangle
    \end{align}
    After minimizing over $h$, we have two possibilities. The case with exact line-search follows immediately after these two cases. In the following, we use the notation $h_k \defas f(x_k) - f^*$ for the primal suboptimality at $x_k$, and $g_k \defas  \langle -\nabla f(x_k),\, \delta(x_k) \rangle$ for the Frank-Wolfe gap at $x_k$ (and note that $g_k \geq h_k$ by convexity).
    
    \textbf{Case 1:} $h_{\text{opt}} = \frac{1}{\tilde{\mathcal{L}}_{f,{\delta}}(x_0)} \sqrt{\frac{f(x_k) - f^*}{f(x_0) - f^*}}$.
    In such case, we obtain (subtract $f^*$ on both sides of the inequality)
    \[
        h_{k+1} \leq h_k - \frac{1}{2\tilde{\mathcal{L}}_{f,{\delta}}}  \sqrt{\frac{h_k}{h_0}} g_k,
    \]
    and since the Frank-Wolfe gap $g_k$ upper bounds the primal suboptimality, we obtain
    \[
    h_{k+1} \leq h_k \Big[ 1 - \frac{1}{2 \tilde{\mathcal{L}}_{f,{\delta}} \sqrt{h_0}} \sqrt{h_k}\Big].
    \]
    
    \textbf{Case 2:} With $h_{\text{opt}} = 1$, we have
    \[
    h_{k+1} \leq  h_k + \left(1- \frac{\mathcal{L}_{f,{\delta}}}{2} \sqrt{\frac{h_0}{h_k}}\right) g_k.
    \]
    In that case, we have that $\frac{1}{\tilde{\mathcal{L}}_{f,{\delta}}(x_0)} \sqrt{\frac{h_k}{h_0}}\geq 1$. Hence we obtain
    \[
         \textstyle h_{k+1} \leq h_k - \frac{1}{2} g_k \leq \frac{1}{2} h_k
    \]

Finally, we have the following recursive relation on the sequence of primal suboptimality $(h_k)$:
\begin{align} 
h_{k+1} &\leq h_k \cdot \max \Big\{ \frac{1}{2}, \, 1 - \frac{1}{2 \tilde{\mathcal{L}}_{f,{\delta}}\sqrt{h_0}} \sqrt{h_k}\Big\} \notag \\
&= h_k \cdot \max\Big\{ \frac{1}{2}, \, 1 - M \sqrt{h_k}\Big\}, \label{eq:GarberRecurrence}
\end{align}
with $M \defas \frac{1}{2 \tilde{\mathcal{L}}_{f,{\delta}}(x_0)\sqrt{h_0}}$. The inequality~\eqref{eq:GarberRecurrence} is exactly the same recurrence that was analyzed by~\citet{garber2015faster} (see their Equation~(7), with the same notation for $M$), where they have shown a $\mathcal{O}(1/K^2)$ convergence rate.
The exact constant is obtained by following the very same proof as \citep{garber2015faster}, \textit{i.e.} proving by induction that there exists $C$ such that $h_k\leq C/(k+2)^2$. The base case $k=0$ can be trivially obtained by letting $C \geq 4 h_0$.\footnote{Note that~\citet{garber2015faster} use a different argument for the base case, bounding instead $h_1$ with $L \cdot \text{diam}(\mathcal{C})^2/2$, using the Lipschitz smoothness of $f$ (and this would become $C_f/2$ in its affine invariant formulation with $C_f$ as defined by~\citet{jaggi2013revisiting}). However, we believe that $h_0$ is usually smaller than $C_f$ in applications, and in any case $h_0$ appears from $1/M^2$ for us, so using our different base case argument is more meaningful.}
Their induction step was shown by requiring that $C\geq \frac{18}{M^2}$. Thus using $C = \max\{4 h_0, \frac{18}{M^2}\}$ (and re-arranging) proves the statement of our theorem.
\end{proof}

The following lemma will be used in the proof of the bound on the modified directional smoothness.

\begin{lemma}\label{lem:strong_convexity_f_quadratic_growth}
Consider a compact convex set $\mathcal{C}$. Assume $f$ is a $\mu_{\omega}$-strongly convex function with respect to $\omega$. Let $x^*$ be the minimum of $f$ on $\mathcal{C}$. Then, for any $x\in\mathcal{C}$, we have
\begin{equation}\label{eq:consequence_strong_convexity}
\omega_{*}(\nabla f(x)) \geq \sqrt{\frac{\mu_{\omega}}{2}} \sqrt{f(x) - f(x^*)}.
\end{equation}
\end{lemma}
\begin{proof}
Let $x\in\mathcal{C}$. From Definition \ref{def:smoothness_strong_convexity_general}, we have that
\[
f(x) \geq f(x^*) + \langle \nabla f(x^*), \, x - x^* \rangle + \frac{\mu_{\omega}}{2} \omega^2(x - x^\star).
\]
Hence with the optimality conditions, \textit{i.e.} $\langle \nabla f(x^*), \, x - x^* \rangle \geq 0$, we have
\begin{equation}\label{eq:intermediaire_strong_convexity}
f(x) - f(x^*) \geq \frac{\mu_{\omega}}{2} \omega^2(x - x^*).
\end{equation}
By convexity of $f$, we have $\langle x - x^*, \, \nabla f(x) \rangle \geq f(x) - f(x^*)$, and by definition of the Fenchel conjugate, we have
\[
\omega(x - x^*) \cdot \omega_{*}(\nabla f(x)) \geq \langle x - x^*, \, \nabla f(x) \rangle \geq f(x) - f(x^*).
\]
Hence by plugging \eqref{eq:intermediaire_strong_convexity}, we obtain \eqref{eq:consequence_strong_convexity}.
\end{proof}

We now prove Theorem \ref{thm:modified_directionnal_smoothness_bound} that is similar to Theorem \ref{thm:directionnal_smoothness_bound}.
It states that in the case of the FW algorithm, the modified directional smoothness constant is bounded if the function is smooth, strongly convex and the set is strongly convex for any distance function $\omega$. It also provides an explicit upper bound on the modified directional smoothness constant. This bound implies that the convergence rate in Theorem \ref{thm:affine_invariant_sublinear_convergence} is better conditioned than existing results \citep{garber2015faster}.

\begin{theorem}[Bounds on modified directional smoothness] \label{thm:modified_directionnal_smoothness_bound}
Consider $x_0\in\mathcal{C}$ and a function $f$, smooth w.r.t.\ the distance function $\omega$, with constant $L_\omega$, strongly convex w.r.t.\ the distance function $\omega$, with constant $\mu_\omega$, and the set $\mathcal{C}$, strongly convex with constant $\alpha_\omega$. Let $\delta(x) = x-v(x)$, $v(x)$ being the FW corner.
Then, the function $f(x)$ is modified directionally smooth w.r.t.\ to $\delta$, with constant
    \begin{equation} \label{eq:bound_ratio_modified_directional_smoothness}
        \tilde{\mathcal{L}}_{f,{\delta}}(x_0) \leq \frac{\kappa_\omega \sqrt{2} L_{\omega}}{\alpha_{\omega}\sqrt{\mu_{\omega}}} \frac{1}{\sqrt{f(x_0)-f^*}}.
    \end{equation}
\end{theorem}
\begin{proof}
Let $h\in[0,1]$. With the smoothness of $f$, we have
\[
f(x + h \delta(x)) \leq f(x) - h \langle -\nabla f(x), \, \delta(x) \rangle + \frac{h^2L_{\omega}}{2} \omega\big(\delta(x)\big)^2.
\]
Recall that when $\delta(x)$ is the Frank-Wolfe direction, we have that the Frank-Wolfe gap $g(x)$ is equal to $\langle -\nabla f(x), \, \delta(x)\rangle$. Also, the scaling inequality for strongly convex sets (Lemma \ref{lem:scaling_inequality}) implies that $\omega(\delta(x))^2 \leq g(x)/(\alpha_{\omega} \omega^\star(-\nabla f(x)))$, so that
\[
f(x + h \delta(x)) \leq f(x) - h \langle -\nabla f(x), \, \delta(x) \rangle + \frac{h^2L_{\omega}}{2 \alpha_{\omega} } \frac{g(x)}{\omega^\star(-\nabla f(x))}.
\]

Now, it is easy to see from the definition of the dual distance $\omega_*$ that is has the same bounded asymmetry constant as for $\omega$, and thus $\omega^\star(-\nabla f(x)) \geq \frac{1}{\kappa_\omega} \omega^\star(\nabla f(x))$.
Thus we apply~\eqref{eq:consequence_strong_convexity} to obtain:
\[
f(x + h \delta(x)) \leq f(x) - h g(x) + \frac{h^2}{2} \frac{\kappa_w \sqrt{2} L_{\omega}}{\alpha_{\omega}\sqrt{\mu_{\omega}}\sqrt{f(x_0)-f^*}} \frac{\sqrt{f(x_0)-f^*}}{\sqrt{f(x)-f^*}} g(x),
\]
which implies equation \eqref{eq:bound_ratio_modified_directional_smoothness}.
\end{proof}

Theorem \ref{thm:modified_directionnal_smoothness_bound} shows that the conditioning of convergence with the directional smoothness, which does not depend on any norm choice, in Theorem \ref{thm:affine_invariant_sublinear_convergence} is better than conditioning of other analysis \citep{garber2015faster}. We now prove that the optimal constant of modified directional smoothness $\tilde{L}_{f,\delta}$ is affine invariant, a result similar to Proposition \ref{prop:affine_invariance_directional_smoothness} for the directional smoothness constant.

\begin{proposition}[Affine Invariance of Modified Directional Smoothness]\label{prop:affine_invariance_modified_directional_smoothness}
Consider $\mathcal{C}$ a compact convex set and $f$ a convex function on $\mathcal{C}$ that is modified directionally smooth w.r.t.\ $\delta(x)$ with constant $\tilde{\mathcal{L}}_{f,\delta}(x_0)$ (with $x_0\in\mathcal{C}$).
If for any $x\in\mathcal{C}$, $\delta(x)$ is affine covariant (e.g.\  the Frank-Wolfe direction $\delta(x) \triangleq v(x)-x$), then the constant $\tilde{\mathcal{L}}_{f,{\delta}}$ in \eqref{eq:modified_directionnal_smoothness} is affine invariant. In other words, for an invertible matrix $B$, let 
\[
    \tilde f(\cdot) \triangleq f(B\cdot), \;\; \tilde \delta_{\tilde{\mathcal{C}}}(\cdot) \triangleq \delta_ {B^{-1}\cdot\mathcal{C}}(\cdot),
\]
then $\tilde{\mathcal{L}}_{\tilde f,\tilde \delta_{\tilde{\mathcal{C}}}}(x_0) = \tilde{\mathcal{L}}_{f,\delta}(y_0)$, where $y_0\triangleq B^{-1}x_0$.
\end{proposition}
\begin{proof}
Let $y\in B^{-1}\cdot \mathcal{C}$. Applying the definition of directional smoothness for $f$ at $B y$, we obtain
\begin{equation}\label{eq:directional_smoothness_B}
f\big(By + h\delta(By)\big) \leq  f(By) + h\langle \nabla f(By),\, \delta(By) \rangle - \frac{\tilde{\mathcal{L}}_{f,\delta}(x_0) h^2}{2}
    \langle \nabla f(By),\, \delta(By) \rangle
    \sqrt{\frac{f(x_0) - f^*}{f(By) - f^*}}.
\end{equation}
Similarly to Proposition \ref{prop:affine_invariance_directional_smoothness}, we have that $\nabla \tilde{f}(y) = B^T \nabla f(By)$ and $\delta(By) = B \tilde{\delta}_{\tilde{\mathcal{C}}}(y)$ so that
\[
\langle \nabla f(By),\, \delta(By) \rangle =  \langle \nabla f(By),\, B \tilde{\delta}_{\tilde{\mathcal{C}}}(y) \rangle
=  \langle B^T\nabla f(By),\, \tilde{\delta}_{\tilde{\mathcal{C}}}(y) \rangle = \langle \nabla \tilde{f}(y), \, \tilde{\delta}_{\tilde{\mathcal{C}}}(y)\rangle. 
\]
Hence \eqref{eq:directional_smoothness_B} and $\tilde{f}^*=f^*$, implies that for any $y\in B^{-1}\cdot\mathcal{C}$
\[
\tilde{f}(y + h \tilde{\delta}_{\tilde{\mathcal{C}}}) \leq \tilde{f}(y) + h \langle \nabla \tilde{f}(y), \, \tilde{\delta}_{\tilde{\mathcal{C}}}(y)\rangle - \frac{\tilde{\mathcal{L}}_{f,\delta}(x_0) h^2}{2}
    \langle \nabla \tilde{f}(y), \, \tilde{\delta}_{\tilde{\mathcal{C}}}(y)\rangle
    \sqrt{\frac{\tilde{f}(y_0) - \tilde{f}^*}{\tilde{f}(y) - \tilde{f}^*}}.
\]
Hence, $\tilde{f}$ is modified directionally smooth on $\tilde{\mathcal{C}}\triangleq B^{-1}\cdot\mathcal{C}$ with respect to $\tilde{\delta}_{\tilde{\mathcal{C}}}$ and
$\tilde{L}_{\tilde{f},\tilde{\delta}_{\tilde{\mathcal{C}}}}(y_0) \leq \tilde{\mathcal{L}}_{f, \delta}(x_0)$.
A similar reasoning concludes that the two constants are equal.
\end{proof}

\section{Related Work Details}

\cite{lacoste2013affine} propose an affine invariant analysis of the vanilla Frank-Wolfe algorithm when the unconstrained optimum $x^*$ is in the relative interior of the constraint set $\mathcal{C}$ and $f$ is strongly convex. Hence, the analysis applies when the constraint set is a strongly convex set, and the quantity might be defined in our context. However, the affine invariant constant $\mu_f^{(FW)}$ standing for the strong convexity of $f$ is zero whenever the optimum is not in the relative interior of the constraint set $\mathcal{C}$.
Indeed, Equation~(3) from~\citep{lacoste2013affine} define the following affine invariant quantity
\[
\mu_f^{(FW)} \triangleq \underset{\substack{x\in\mathcal{C}\setminus\{x^*\}, \gamma\in]0,1]\\ \bar{s} = \bar{s}(x, x^*, \mathcal{C})\\ y = x + \gamma (\bar{s} - x)}}{\text{inf }} \frac{2}{\gamma^2} \big[f(y) - f(x) - \langle \nabla f(x),\, y-x \rangle \big],
\]
where $\bar{s}(x, x^*, \mathcal{C}) = \text{ray}(x, x^*)\cap \partial\mathcal{C}$. When $x^*\notin\mathcal{C}$, we have $\mu_f^{(FW)} \leq 0$ since there are some point $x\in\partial\mathcal{C}$ such that $x \in \bar{s}(x, x^*, \mathcal{C})$, and thus we can take $\bar{s}=x$ in the $\inf$, yielding $y=x$ with $\gamma>0$. This means that the above quantity cannot be easily generalized to the setting we studied in Theorem~\ref{thm:directionnal_smoothness_bound} where the unconstrained optimum is assumed to be \emph{outside} of $\mathcal{C}$.
\end{document}

%% file: utils/defs.tex
\newtheorem{theorem}{Theorem}[section]
\newtheorem{proposition}[theorem]{Proposition}

\newtheorem{definition}[theorem]{Definition}
\newtheorem{remark}[theorem]{Remark}
\newtheorem{lemma}[theorem]{Lemma}

\newtheorem{example}[theorem]{Example}
\newenvironment{proof}{\textbf{Proof. }}{\QED\bigskip}
\newtheorem{assumption}[theorem]{Assumption}

\def\xx{{\boldsymbol x}}

\newcommand{\BEAS}{\begin{eqnarray*}}
\newcommand{\EEAS}{\end{eqnarray*}}
\newcommand{\BEA}{\begin{eqnarray}}
\newcommand{\EEA}{\end{eqnarray}}
\newcommand{\BEQ}{\begin{equation}}
\newcommand{\EEQ}{\end{equation}}
\newcommand{\BIT}{\begin{itemize}}
\newcommand{\EIT}{\end{itemize}}
\newcommand{\BNUM}{\begin{enumerate}}
\newcommand{\ENUM}{\end{enumerate}}

\newcommand{\BA}{\begin{array}}
\newcommand{\EA}{\end{array}}

\newcommand{\QED}{~~\rule[-1pt]{6pt}{6pt}}

\newcommand{\argmin}{\mathop{\rm argmin}}

\newcommand{\argmax}{\mathop{\rm argmax}}



%% file: utils/preamble.tex
\usepackage{amsmath}
\usepackage{amssymb}
\usepackage{empheq}
\usepackage{xcolor, color, colortbl}
\usepackage{framed}
\usepackage{pifont}
\newcommand{\cmark}{\ding{51}}%
\newcommand{\xmark}{\ding{55}}%
\usepackage{enumitem}
\usepackage{graphicx} 

\usepackage{algorithmic,algorithm}

\usepackage{nicefrac}
\usepackage{booktabs}
\usepackage{multirow}
\usepackage{rotating}
\usepackage{nccmath}

\usepackage[tikz]{bclogo}
\usepackage{wasysym}

\definecolor{darkorange}{rgb}{0.85, 0.45, 0}

\definecolor{mydarkblue}{rgb}{0,0.08,0.45}
\usepackage[
    colorlinks=true,
    linkcolor=mydarkblue,
    citecolor=mydarkblue,
    filecolor=mydarkblue,
    urlcolor=mydarkblue,
    allcolors=mydarkblue,
    pdfview=FitH]{hyperref}

\usepackage{url}
\usepackage{relsize}

\def\xx{{\boldsymbol x}}

\renewcommand{\vec}{\mathbf{vec}}

\def\defas{\stackrel{\text{def}}{=}}

\definecolor{myblue}{HTML}{D2E4FC}
\definecolor{Gray}{gray}{0.92}

\usepackage{thmtools}
\usepackage{thm-restate}

\usepackage{wrapfig}

\newcommand{\cC}{\mathcal{C}}

%% file: utils/maths_commands.tex
\usepackage{amsmath,amsfonts,bm}

\def\1{\bm{1}}

\DeclareMathAlphabet{\mathsfit}{\encodingdefault}{\sfdefault}{m}{sl}
\SetMathAlphabet{\mathsfit}{bold}{\encodingdefault}{\sfdefault}{bx}{n}

